\documentclass[10pt]{amsart}
\usepackage{graphicx}
\usepackage{amssymb}
\usepackage{amsmath}
\def\be{\begin{equation}}
\def\ee{\end{equation}}

\newtheorem{lemma}{Lemma}[section]

\newtheorem{theorem}[lemma]{Theorem}

\catcode`\@=11

\def\Expect{\mathbb{E}~}
\def\Prob{\mathbb{P}}

\newenvironment{dedication}
        {\vspace{6ex}\begin{quotation}\begin{center}\begin{em}}
        {\par\end{em}\end{center}\end{quotation}}

\title{On Busy Periods of the Critical GI/G/1 Queue and BRAVO}

\title{On Busy Periods of the Critical GI/G/1 Queue and BRAVO}

\author{Yoni Nazarathy}\address{School of Mathematics and Physics, The University of Queensland, Brisbane, QLD 4072, Australia}
\email{y.nazarathy@uq.edu.au}

\author[Z. Palmowski]{Zbigniew Palmowski}
\address{ Faculty of Pure and Applied Mathematics,
Wroc\l aw University of Science and Technology,
Wyb. Wyspia\'nskiego 27, 50-370 Wroc\l aw, Poland}
\email{zbigniew.palmowski@pwr.edu.pl}
\thanks{This work is partially supported by Polish National Science Centre Grant No. 2018/29/B/ST1/00756, 2019-2022}

\date{ \today}

\begin{document}
\maketitle

\begin{dedication}
%%\hspace{0.5cm}
%%\vspace*{9cm}
This paper is dedicated to Masakiyo Miyazawa in appreciation of his fundamental contributions
to applied probability
\end{dedication}
\begin{abstract}
We study critical GI/G/1 queues under finite second moment assumptions. We show that the busy period distribution is regularly varying with index half. We also review previously known M/G/1/ and M/M/1 derivations, yielding exact asymptotics as well as a similar derivation for GI/M/1. The busy period asymptotics determine the growth rate of moments of the renewal process counting busy cycles.  We further use this to demonstrate a BRAVO phenomenon (Balancing Reduces Asymptotic Variance of Outputs) for the work-output process (namely the busy-time).  This yields new insight on the BRAVO effect.

A second contribution of the paper is in settling previous conjectured results about GI/G/1 and GI/G/s BRAVO. Previously, infinite buffer BRAVO was generally only settled under fourth-moment assumptions together with an assumption about the tail of the busy-period. In the current paper we strengthen the previous results by reducing to assumptions to existence of $2+\epsilon$ moments.
\end{abstract}

%%%%%%%%%%%%%%%%%%%%%%%%%%%%%%%%%%%%%%%%%%%%%%%%
%%%%%%%%%%%%%%%%%%%%%%%%%%%%%%%%%%%%%%%%%%%%%%%%
%%%%%%%%%%%%%%%%%%%%%%%%%%%%%%%%%%%%%%%%%%%%%%%%
%%%%%%%%%%%%%%%%%%%%%%%%%%%%%%%%%%%%%%%%%%%%%%%%
%%%%%%%%%%%%%%%%%%%%%%%%%%%%%%%%%%%%%%%%%%%%%%%%
%%%%%%%%%%%%%%%%%%%%%%%%%%%%%%%%%%%%%%%%%%%%%%%%

\section{Introduction}
This paper deals with the well studied GI/G/1 queue. In many ways this stochastic model lives in the centre of queueing theory and applied probability as it marks the border of tractable explicit models (e.g. M/G/1) and models for which asymptotic approximations are needed. Living on the explicitly intractable side of the border, GI/G/1 has motivated much early works in asymptotic approximate queueing theory such as \cite{kingman1962queues}, \cite{borovkov1964some} and \cite{iglehart1970mcq}. Indeed GI/G/1 is an extremely studied model (see \cite[Chapter X]{bookAsmussen2003} for an overview). Nevertheless, in this paper, we add new results to the body of knowledge about GI/G/1.

There are several stochastic processes and random variables associated with GI/G/1. Of key interest to us is the the queue length process (including the customer in service), $\{Q(t),~t\ge0\}$, the departure counting process (indicating the number of service completions during $[0,t]$), $\{D(t),~t \ge0\}$, the busy period random variable, $B$, as well as several other processes which we describe in the sequel. The processes $Q(\cdot)$, $D(\cdot)$ and the random variable $B$ are constructed in a standard way (see e.g. \cite[Chapter X]{bookAsmussen2003}) on a probability space supporting two independent i.i.d. sequences of strictly positive random variables. Namely, $\{U_i\}_{i=1}^\infty$ denote the inter-arrival times and $\{V_i\}_{i=1}^\infty$ denotes the service times. Note that in this paper we assume that an arrival to an empty system occurs at time $0$ yielding service duration $V_1$ and then after time $U_1$ the next arrival occurs with service duration $V_2$ and so forth.

As in many, but not all, of the GI/G/1 studies, our focus is on the very special {\em critical} case, i.e. we assume,
\begin{equation}\label{critical}
\Expect[U_1] = \Expect[V_1]:=\lambda^{-1}.
\end{equation}
We also assume that $\Expect[U_1^2],\, \Expect[V_1^2] < \infty$, that is, we are in the case where the inter-arrival and service sequences obey a Gaussian central limit law.

Critical GI/G/1 has been an exciting topic for research for many reasons. It lies on the border of stability ($\Expect[U_1] > \Expect[V_1]$) and instability ($\Expect[U_1] < \Expect[V_1]$). In the near-critical but stable case, sojourn time random variables are approximately exponentially distributed (under some regularity assumptions). Further the queue length and/or workload  processes converge to reflected Brownian motion when viewed through the lens of diffusion scaling. Such scaling can be done both in the near-critical case or exactly at criticality. This is then a much more tractable object for further analysis, in comparison to the associated random walks that are used in the non-critical case. See for example \cite{bookWhitt2001} for a comprehensive treatment of diffusion scaling limits. A useful survey is also in \cite{Glynn0518}. So in general, analysis in the critical case has attracted much attention.

A further specialty arising in the critical case is the well-known fact that the while the busy-period random variable is finite w.p. $1$, it has infinite expectation. One of the contributions of this paper is that we further establish (under the aforementioned finite-second moment assumptions), that the busy period is a regularly varying random variable with index $1/2$, where the associated slowly varying function is bounded from below.  This property of the critical busy period is well-known and  essentially elementary to show for the M/M/1 case. Further, for the M/G/1 case it was established as a side result in \cite{zwart2001tail} with exact tail asymptotics. One of the main result of the current paper is that we find similar exact asymptotics for the GI/G/1 case.  Tails and related properties of the busy period for GI/G/1 and related models have received much attention during the past 15 years. In addition to \cite{zwart2001tail}, the busy period has been studied under various conditions in \cite{baltrunas2004tail}, \cite{kim2007tail}, \cite{denisov2010global}, \cite{baltrunas2002second}, \cite{liu2014tail}, \cite{palmowski2006exact}, \cite{denisov2013asymptotics} and \cite{foss2004existence}.
Our contribution to the body of knowledge in the critical case adds to this.

One of the important reasons for studying the tail behavior of the critical busy-period is in relation to the BRAVO effect (Balancing Reduces Asymptotic Variance of Outputs). This is a phenomenon occurring in a variety of queueing systems (see for example \cite{nazarathy2011variance}), and specifically for the critical GI/G/1. For systems without buffer limitations such as GI/G/1, it was first presented in \cite{al2011asymptotic}. The (somewhat counterintuitive) BRAVO effect is that the long-term variability of the output counts is less than the average variability of the arrival and service process. This is with respect to the {\em asymptotic variance},
\[
\overline{v} := \lim_{t \to \infty} \frac{var\big(D(t)\big)}{t}
= \lambda (c_a^2+c_s^2)  \big(1- \frac{2}{\pi} \big),
\]
where $c_a^2$ and $c_s^2$ are the squared coefficient of variation of the building block random variables, i.e.,
\[
c_a^2 := \frac{var(U_1)}{\Expect[U_1]^2},
\qquad
c_s^2 := \frac{var(V_1)}{\Expect[V_1]^2}.
\]
It thus follows that the variability function  $\lim_{t\to\infty} var D(t)/\mathbb{E} D(t)$, when considered as a function of the system load, has a singular point at the system load equal $1$, which can be regarded as a manifestation of the BRAVO phenomenon.
More specifically, $\overline{v}$ is essentially determined by either the arrival or the service
process when for the system load is not equal $1$, whereas for the system load equal $1$
it is determined by both the arrival and service processes.

In \cite{al2011asymptotic}, the BRAVO effect for GI/G/1 was established for M/M/1 by first principles and for the GI/G/1 it was derived via a classic diffusion limits for $D(\cdot)$,  \cite{iglehart1970mcq}. Following the development of  \cite{al2011asymptotic}, a key technical component for a complete proof of GI/G/1 BRAVO is uniform integrability (UI) of the sequence,
\[
\Big\{
\frac{Q(t)^2}{t},~t \ge t_0
\Big\},
\]
for some non-negative $t_0$.  This UI property was only established in  \cite{al2011asymptotic} under the simplifying assumption that $\Expect[U_1^4] ,~ \Expect[V_1^4] < \infty$ and further under the assumption that the tail behaviour of the busy-period was regularly varying with index half (as we establish in the current paper).

It was further conjectured in \cite{al2011asymptotic} that GI/G/1 BRAVO persists under the more relaxed assumptions that we consider here.
By using new tail behaviour results for the busy period and
further considering properties of renewal processes that generalise some results in \cite{al2011asymptotic}, we settle this conjecture in the current paper under weaker assumptions of existence of $2+\epsilon$ moments of generic inter-arrival and service time.

The structure of the remainder of the paper is as follows. In Section~\ref{sec:busyPeriod} we study the busy period tail behaviour
and compare it with the M/M/1 and M/G/1 cases.
%In Section~\ref{sec:renewal} we handle the renewal process of busy cycles.
In Section~\ref{sec:bravoAgain} we establish GI/G/1 BRAVO under the relaxed assumptions.
%We conclude in Section~\ref{sec:conclusion}.

\section{Busy Period}
\label{sec:busyPeriod}
For analysis of the busy period, it is useful to denote $\xi_i := V_i-U_{i}$ for $i=1,2,\ldots$ and
\begin{equation}\label{rw}
S_n := \sum_{i=1}^n \xi_i\quad \text{with $S_0=0$.}
\end{equation}
 The random walk $S_n$ is embedded within the well known workload process. Within the first busy period, $S_n$ denotes the workload of the system immediately after the arrival of customer $n$. Then the number of customers served during this busy period is $N := \inf \{ n \ge 1 : S_n \le 0\}$.  With $N$ at hand we can then define the busy period (duration) random variable as,
\begin{equation}\label{busyrepr}
B := \sum_{i=1}^N V_i.
\end{equation}
Note of course that $N$ and the sequence $\{V_i\}$ are generally dependent and further note that $N$ is a stopping time with infinite expectation.
We will also need the generic idle period that follows the busy period which equals
\[
I := -S_N.
\]

We recall that in this paper we assume the critical case when \eqref{critical} holds.
we will write $f(x)\sim g(x)$ when $\lim_{x\to\infty} f(x)/g(x)=1$.
The main result of this section is given by the following  theorem
\begin{theorem}
\label{thm:bpMain}
Consider the case where $\Expect[U_1] = \Expect[V_1]$ and $\Expect[U_1^2],~ \Expect[V_1^2]< \infty$. Then,
\[
\Prob \big( B > x \big) \sim \mathbb{E}I \sqrt{\frac{2\lambda}{\pi(c_a^2+c_s^2)}} x^{-1/2}.
\]
\end{theorem}
\begin{proof}
From \cite[Thm. XVIII.5.1, p. 612]{feller1966ipt} we know that the series
\[\sum_{n=1}^\infty\left[\mathbb{P}(S_n<0)-\frac{1}{2}
\right]\]
converges to some $b>0$ and
\begin{equation}\label{meanidle}
\mathbb{E}I=-\mathbb{E}S_N=\frac{\sigma}{\sqrt{2}}e^{-b}
\end{equation}
where
\[\sigma^2=var(V_1-U_1)=(c_a^2+c_s^2)\lambda^{-2}.\]
Moreover, from \cite[Thm. XII.7.1a, p. 415]{feller1966ipt}, we have
\begin{equation}\label{asympN}
\mathbb{P}(N>n)\sim \frac{1}{\sqrt{\pi}}e^{-b}n^{-1/2}.
\end{equation}
Combining above facts gives
\begin{equation}\label{asympN}
\mathbb{P}(N>n)\sim -\mathbb{E}S_N\sqrt{\frac{2}{\pi\sigma ^2}}n^{-1/2}=
\mathbb{E}I \sqrt{\frac{2\lambda^2}{\pi (c_a^2+c_s^2)}}n^{-1/2}.
\end{equation}
In the last step we apply \cite[Thm. 3.1]{robert2008tails} together with the representation \eqref{busyrepr}.
Namely, observe that by \eqref{asympN} is regularly varying and hence of consistent variation.
Moreover, $\mathbb{E}V_1^2<\infty$ and $x\mathbb{P}(V_1>x)=o(\mathbb{P}(N>x))$.
Thus
\[\mathbb{P}(B>x)\sim \mathbb{P}(N>\lambda x)\]
which completes the proof in view of \eqref{busyrepr}.
\end{proof}

From Theorem \ref{thm:bpMain} one can recover known
result for the M/G/1 queue.  Indeed, in this case the idle period has an exponential distribution with the parameter $\lambda>0$ and therefore
$\mathbb{E}I=\frac{1}{\lambda}$. Furthermore, $c_a=1$ and hence
\[\Prob(B > x ) \sim \lambda ^{-1/2} \sqrt{\frac{2}{( 1 + c_s^2) \pi }}{x}^{-1/2}
\]
The M/G/1 result appeared in \cite{zwart2001tail}, but with a small typo for the constant in front of $x^{-1/2}$.

In the case of G/M/1 queue, we have $c_s=1$ and the first increasing ladder height $H_1$ of the random walk \eqref{rw}
has exponential distribution with the parameter $\lambda>0$ and therefore $\mathbb{E}H_1=\frac{1}{\lambda}$.
Moreover, by \cite[eq. (4c)]{Doney1980} we have
\[(c_a^2+c_s^2)\lambda^{-2}=2\mathbb{E}H_1\mathbb{E}I=\frac{2}{\lambda}\mathbb{E}I\]
and therefore
\[\mathbb{E}I=(c_a^2+c_s^2)\frac{1}{2\lambda}.\]
Theorem \ref{thm:bpMain} gives then
\[
\Prob(B > x ) \sim
\lambda ^{-1/2} \sqrt{\frac{( c_a^2+1)}{2 \pi}}{x}^{-1/2}.
\]
Finally we mention that the above agree with the M/M/1 case with
$\Prob(B > x ) \sim
\lambda ^{-1/2} {(\pi x)}^{-1/2}$ which may be also obtained via asymptotics of Bessel functions as the distribution of $B$ is exactly know.

\section{BRAVO}
\label{sec:bravoAgain}

We now establish BRAVO for the GI/G/1 with what we believe to be the minimal possible set of assumptions. The result below unifies Theorem~2.1 and Theorem~2.2 of  \cite{al2011asymptotic} and also settles Conjecture 2.2 of that paper for the single-server case.
\begin{theorem}
\label{thm:mainBRAVO}
Consider the GI/G/1 queue with $\Expect[U_1] = \Expect[V_1]$ and $\Expect[U_1^{2+\epsilon}],~ \Expect[V_1^{2+\epsilon}]< \infty$ for any $\epsilon>0$. Then,
\[
\lim_{t \to \infty} \frac{Var\big(D(t)\big)}{\Expect[D(t)]} =  (c_a^2+c_s^2)  \big(1- \frac{2}{\pi} \big).
\]
\end{theorem}

\begin{proof}
Following Theorem~2.1 of \cite{al2011asymptotic} we need to show uniformly integrability (UI) of
\[
\mathcal{Q} := \Bigg\{  \frac{Q(t)^2}{t},~t \ge t_0 \Bigg\}
\]
for some $t_0>0$.
Once we establish this UI, the result follows.

We can follow the same idea like one can find
in the proof of \cite[Thm. 2.2]{al2011asymptotic}. Namely, we
can apply \cite[Lem. 2.1]{al2011asymptotic} with $r=2+\epsilon$ for $\epsilon>0$
and \cite[Prop. 4.1, Thm. 4.1, Thm. 4.2(ii) and Thm. 4.3]{al2011asymptotic} exchanging 4th moments by $2+\epsilon$ ones
and $8=2^3$ should be replaced by $2^{1+\epsilon}$.
There is one crucial difference though in the proof of \cite[Thm. 4.1]{al2011asymptotic}.
Namely, without using the Wald idenity used in \cite{al2011asymptotic},
we have to show there that, $\Expect[M_{V(t)}^{2+\epsilon}] = O(t^{1+\epsilon/2})$, where
$M_n = \sum_{i=1}^n \frac{1}{\Expect[\zeta_i]} \zeta_i - n$ is mean zero random walk with
$\zeta_i$ being i.i.d. random variables with $2+\epsilon$ finite moments
and $V(t) = \inf \{n : \sum_{i=1}^n \zeta_i \ge t \}$ is the first passage time.

From the Marcinkiewicz-Zygmund inequality for a stopped random walk
given in \cite[Thm. I.5.1(iii), p. 22]{Gut} with $r=2+\epsilon$ we can conclude that
\[\mathbb{E} |M_{V(t)}|^{2+\epsilon}\leq K\; \mathbb{E}\left|\zeta_1-\frac{1}{\Expect[\zeta_i]}\right|^{2+\epsilon}
\mathbb{E}V(t)^{(2+\epsilon)/2}=
{\rm O}\left(\mathbb{E}V(t)^{1+\epsilon/2}\right)\]
for some constant $K$.
%
%To do so, observe that $M_k$ is discrete time martingale.
%From Skorokhod embedding problem
%we know that there exists stopping time $\tau$ such that for a Brownian motion random variable $B(\tau)$ has a law of
%$\frac{1}{\Expect[\zeta_i]} \zeta_i - 1$ where we assume that $E\zeta_i^{2+\epsilon}<\infty$ for $\epsilon>0$.
%Denote by $\tau_i$ consecutive independent stopping times producing above generic random variable.
%Then $M_n=B(\tau_n)$.
%Moreover, by \cite[Prop. 1]{Jan} and Optional Stopping Theorem we know that there exists constant
%$C$ such that
%\begin{equation}\label{upperbound}
%E |M_n|^{2+\epsilon}=EB(\tau_n)^{2+\epsilon}\le C E\tau_n^{1+\epsilon/2}.
%\end{equation}
%Now using classical regeneration arguments (see eg. \cite[Prop. 5.1.4]{bookAsmussen2003}) we can conclude that
%\[E |M_n|^{2+\epsilon}={\rm O}(n^{1+\epsilon/2})\]
%as $n\to\infty$.
Moreover, from \cite[Thm. III.8.1, p. 98]{Gut} we know that
$\lim_{t\to\infty} \mathbb{E}V(t)^{1+\epsilon/2}/t^{1+\epsilon/2} <\infty$. Hence
\[\mathbb{E} |M_{V(t)}|^{2+\epsilon}={\rm O}(t^{1+\epsilon/2})\]
which completes the proof.
\end{proof}

%\section{Conclusion}
%\label{sec:conclusion}

%\bibliography{PaperDatabase,BookDatabase}

\begin{thebibliography}{10}

\bibitem{al2011asymptotic}
A. Al Hanbali, M. Mandjes, Y. Nazarathy, and W. Whitt.
\newblock The asymptotic variance of departures in critically loaded queues.
\newblock {\em Advances in Applied Probability}, 43(1):243--263, 2011.

\bibitem{bookAsmussen2003}
S. Asmussen.
\newblock {\em Applied Probability and Queues}.
\newblock Springer-Verlag, 2003.

\bibitem{baltrunas2002second}
A. Baltr{\=u}nas.
\newblock Second-order tail behavior of the busy period distribution of certain
  GI/G/1 queues.
\newblock {\em Lithuanian Mathematical Journal}, 42(3):243--254, 2002.

\bibitem{baltrunas2004tail}
A. Baltr{\=u}nas, D.J.~Daley, and C. Kl{\"u}ppelberg.
\newblock Tail behaviour of the busy period of a gi/gi/1 queue with
  subexponential service times.
\newblock {\em Stochastic Processes and their Applications}, 111(2):237--258,
  2004.

\bibitem{borovkov1964some}
A.~A. Borovkov.
\newblock Some limit theorems in the theory of mass service.
\newblock {\em Theory of Probability \& Its Applications}, 9(4):550--565, 1964.

\bibitem{denisov2010global}
D. Denisov and S. Shneer.
\newblock Global and local asymptotics for the busy period of an M/G/1 queue.
\newblock {\em Queueing Systems}, 64(4):383--393, 2010.

\bibitem{denisov2013asymptotics}
D. Denisov, V. Shneer.
\newblock Asymptotics for the first passage times of L{\'e}vy processes and
  random walks.
\newblock {\em Journal of Applied Probability}, 50(1):64--84, 2013.

\bibitem{Doney1980}
R.~A. Doney.
\newblock Moments of ladder heights in random walks.
\newblock {\em Journal of Applied Probability}, 17(1):248--252, 1980.

\bibitem{feller1966ipt}
W. Feller.
\newblock {\em An Introduction to Probability Theory and its Applications.
Volume~2.}
\newblock 1966.

\bibitem{foss2004existence}
S. Foss and A. Sapozhnikov.
\newblock On the existence of moments for the busy period in a single-server
  queue.
\newblock {\em Mathematics of Operations Research}, 29(3):592--601, 2004.

\bibitem{Glynn0518}
P.~W. Glynn.
\newblock Diffusion approximations.
\newblock {\em In Handbooks in Operations Research, Vol. 2, D.P. Heyman and M.J.
  Sobel (eds.), North-Holland, Amsterdam}, 145--198, 1990.

\bibitem{Gut}
A. Gut.
\newblock {\em Stopped Random Walks. Limit Theorems and Applications}.
\newblock Sec. Ed., Springer-Verlag, 2009.

\bibitem{iglehart1970mcq}
D.~L. Iglehart and W. Whitt.
\newblock {Multiple Channel Queues in Heavy Traffic. I}.
\newblock {\em Advances in Applied Probability}, 2(1):150--177, 1970.

\bibitem{kim2007tail}
B. Kim, J. Lee, and I. Wee.
\newblock Tail asymptotics for the fundamental period in the MAP/G/1 queue.
\newblock {\em Queueing Systems}, 57(1):1--18, 2007.

\bibitem{kingman1962queues}
J.~F.~C. Kingman.
\newblock On queues in heavy traffic.
\newblock {\em Journal of the Royal Statistical Society. Series B
  (Methodological)}, 383--392, 1962.

\bibitem{liu2014tail}
B. Liu, J. Wang, and Y.~Q. Zhao.
\newblock Tail asymptotics of the waiting time and the busy period for the
  M/G/1/K queues with
  subexponential service times.
\newblock {\em Queueing Systems}, 76(1):1--19, 2014.

\bibitem{nazarathy2011variance}
Y. Nazarathy.
\newblock The variance of departure processes: puzzling behavior and open
  problems.
\newblock {\em Queueing Systems}, 68(3-4):385--394, 2011.

\bibitem{palmowski2006exact}
Z. Palmowski and T. Rolski.
\newblock On the exact asymptotics of the busy period in GI/G/1 queues.
\newblock {\em Advances in Applied Probability}, 38(3):792--803, 2006.

\bibitem{robert2008tails}
C.~Y. Robert and J. Segers.
\newblock Tails of random sums of a heavy-tailed number of light-tailed terms.
\newblock {\em Insurance: Mathematics and Economics}, 43(1):85--92, 2008.

\bibitem{bookWhitt2001}
W. Whitt.
\newblock {\em {Stochastic Process Limits}}.
\newblock Springer New York, 2002.

\bibitem{zwart2001tail}
A.~P. Zwart.
\newblock {Tail asymptotics for the busy period in the GI/G/1 queue}.
\newblock {\em Mathematics of Operations Research}, 26(3):485--493, 2001.

\end{thebibliography}

\section*{Acknowledgements}
%The authors are very thankful to two referees for their really valuable comments and suggestions.
The authors are very thankful to Peter Taylor that was involved in early discussions on this paper.

\end{document}